\documentclass[10pt,reqno]{amsart}

\usepackage{amsmath,amssymb,amsthm,amsfonts,latexsym,epsfig,enumerate}

\newtheorem{theorem}{Theorem}
\newtheorem{lemma}[theorem]{Lemma}

\newtheorem{proposition}[theorem]{Proposition}
\newtheorem{corollary}[theorem]{Corollary}

\numberwithin{theorem}{section}

\theoremstyle{definition}
\newtheorem*{rmk}{Remark}

\def\f{\frac}
\def\ft#1#2{(#1)/#2}
\def\fb#1#2{#1/(#2)}

\begin{document}

\title{A Sum Involving the Greatest-Integer Function}

\author[David Ross Richman]{David Ross Richman$^1$}
\thanks{$^1$David Richman died on February 1, 1991, in an airplane accident.  
This paper was in his files and has been slightly modified for publication 
by Michael Filaseta, Jeffrey Lagarias, and Harry Richman   
with the consent of David's wife Shumei Cheng Richman.}

\address{Department of Mathematics, University of South Carolina, Columbia, SC 29208, USA}


\maketitle

\begin{abstract}
We determine properties of
the set of values of 
$ [nx] - ([x]/1 + [2x]/2 + \cdots + [nx]/n)$
as $n$ and $x$ vary.
\end{abstract}

\section{Introduction}
\label{sec:intro}
Let $x$ denote a real number and let $n$ denote a positive integer.  Problem
5 of the 1981 U.S.A. Mathematical Olympiad was to prove that
\begin{equation}
\label{eq:olympiad}
[nx] \geq \f {[x]}{1} + \frac{[2x]}{2} + \frac{[3x]}{3} +
\cdots + \frac{[nx]}{n}
\end{equation}
where $[t]$ denotes the greatest integer less than or equal to $t$.  Observe
that 
\begin{equation}
\label{eq:obs}
nx = \frac{x}{1} + \frac{2x}{2} + \frac{3x}{3} + \cdots +
\f{nx}{n} \geq \f{[x]}{1} + \f{[2x]}{2} + \f{[3x]}{3} + \cdots +
\f{[nx]}{n}.
\end{equation}
This relation does not, however, obviously imply \eqref{eq:olympiad}, 
because the sum on the
right-hand side of \eqref{eq:obs} is not necessarily an integer.  
Proofs of \eqref{eq:olympiad} are given by
Klamkin \cite[pp. 92--92]{kl}
and Larsen \cite[p. 279]{l}.

More can be said about relation \eqref{eq:olympiad}; for example,
if equality does not hold 
($[nx] \neq 
[x]/1 + [2x]/2 + \cdots + [nx]/n$),
then in fact
\begin{equation*}
[nx] \geq \frac16 +  \f {[x]}{1} + \frac{[2x]}{2} + \frac{[3x]}{3} +
\cdots + \frac{[nx]}{n}.
\end{equation*}
(This is the content of Proposition~\ref{prop:21}.)

Let 
$f_n(x) 
:= [nx] - ([x]/1 + [2x]/2 + \cdots + [nx]/n)$. 
Let $S_n$ denote the
range of this function; it is a finite set of  rational numbers.
Let $S = \bigcup_{n=1}^{\infty} S_n$; it is a countable set of rational numbers
and relation \eqref{eq:olympiad} is equivalent to the
statement that the elements of $S$ are all nonnegative.  
The main result of this paper is:\medskip

\begin{theorem}
\begin{enumerate}[(i)]
\item The smallest limit point of the set $S$ is
\[
\lambda = \sum^\infty_{k=1} \f{1}{2k(2k+1)} = 1- \log 2 \approx 0.30685 .
\]

\item The members of $S$ smaller than $\lambda$ are given by  
$0$, $\frac{4}{15}$,
and all the partial sums $t_m = \sum_{k=1}^m \frac{1}{2k(2k+1)}$ for $m \ge1.$

\item The members of $S$ larger than  $\lambda$ are dense in the interval
$[ \lambda , +\infty).$
\end{enumerate}
\end{theorem}
The theorem can be summarized in the equivalent form
\begin{align}
\label{eq:summary}
S = \ &\left(\text{the set of partial sums of}\quad \sum^{\infty}_{k=1} 
\f{1}{2k(2k+1)}\right) \ \bigcup \ \bigg\{ 0, \f{4}{15} \bigg\} \nonumber \\ 
&\bigcup \ \left(\text{a dense subset of the interval}
\quad\bigg(\sum^\infty_{k=1} \f{1}{2k(2k+1)}, \infty\bigg)\right).
\end{align}
In particular, \eqref{eq:summary}  implies that 
$[nx] - ([x]/1 + [2x]/2 + \cdots + [nx]/n)$ 
equals $0$ or $\f{1}{6}$ or
$\f{1}{6} + \f{1}{20}$ or $\f{1}{6} + \f{1}{20} + \f{1}{42}$ 
or a number which is greater than or equal to 
$\f{1}{6} + \f{1}{20} + \f{1}{42} +\f{1}{72}$.

\begin{rmk}
Note that $[nx] - ([x]/1 + [2x]/2 + \cdots + [nx]/n) = 0$ 
when $x$ is an integer (or, more generally, when $x - [x] < 1/n$).  
Let $m$ denote a positive integer; 
one can easily
prove by induction on $m$ that if $n = 2m+1$ and $x = 1/(m+1)$, 
then $[nx] - ([x]/1 + [2x]/2 + \cdots + [nx]/n) = \sum^m_{k=1}
1/(2k(2k+1))$.  
If $n = 5$ and $x = 1/2$, then 
$[nx] - ([x]/1 + [2x]/2 + \cdots + [nx]/n) = 4/15$.  
These observations already imply that $S$ contains 
$0, \f{4}{15}$ and all the partial sums of 
$\sum^\infty_{k=1} 1/(2k(2k+1))$.   
\end{rmk}

The author was able to discover \eqref{eq:summary} largely because, 
when $n$ is fixed and $x$ varies, the values of 
$[nx] - ([x]/1 + [2x]/2 + \cdots + [nx]/n)$ 
can be calculated explicitily 
(a similar observation is made in \cite[p. 92]{kl}).  
For example, considering the case that $n = 3$, one has 
\[
[3x] - \f{[x]}{1} - \f{[2x]}{2} - \f{[3x]}{3} = 
\begin{cases}
 0 &\quad\text{when}\quad x - [x] < \f{1}{3}, \\[2pt]
\f{2}{3} &\quad \text{when}\quad \f{1}{3} \leq x-[x] < \f{1}{2},\\[2pt]
\f{1}{6} &\quad \text{when}\quad \f{1}{2} \leq x - [x] < \f{2}{3},\\[2pt]
\f{5}{6} &\quad \text{when}\quad \f{2}{3} \leq x - [x].
\end{cases}
\]
Thus $S_3= \{ 0, \frac{1}{6}, \frac{2}{3}, \frac{5}{6} \}$.
Calculations of this kind are useful for suggesting patterns and 
conjectures, but are not needed to prove \eqref{eq:olympiad} or \eqref{eq:summary}.

To prove \eqref{eq:summary}, 
this paper will focus attention on the smallest number $y$
satisfying $[ky] = [kx]$ for every $k$ in $\{1, \dots ,n\}$, 
when $n$ and $x$ are fixed.
This approach, or a similar idea, is also used in 
\cite[p. 92]{kl} and \cite[p. 279]{l}.

This paper is organized as follows. 
Section~\ref{sec:lower-bound} contains 
a new, simple proof of \eqref{eq:olympiad}.  
In Sections~\ref{sec:prelim} and \ref{sec:limit-point}
we obtain necessary and sufficient conditions for 
$[nx] - ([x]/1 + [2x]/2 + \cdots + [nx]/n)$ 
to be less than $\lambda = \sum^\infty_{k=1} 1/(2k(2k+1))$.  
We will then establish the main result 
\eqref{eq:summary}
in Section~\ref{sec:main-proof}.  
Finally, Section~\ref{sec:upper-bound}
contains a proof that 
$[nx] - ([x]/1 + [2x]/2 + \cdots + [nx]/n) \leq \f{1}{2} + \f{1}{3} + \cdots + \f{1}{n}$. 
This gives an upper bound for $S_n$ which complements
the lower bound for $S_n$ implied by \eqref{eq:olympiad}.

\section{A lower bound for $S_n$}
\label{sec:lower-bound}
We begin by sketching a new proof of the Olympiad problem \eqref{eq:olympiad},
which we restate below.
Recall that $S_n$ denotes the set of numbers of the form
$ [nx] - \sum^{n}_{k=1} {[kx]}/{k}$
where $x$ varies over all real numbers.
\begin{lemma}[1981 USAMO, Problem 5]
For any positive integer $n$ and any $x$, 
\begin{equation*}
[nx] - \sum_{k=1}^n \frac{[kx]}{k} \geq 0 .
\end{equation*}
\end{lemma}
\begin{proof}
Let $n$ be a fixed positive integer.
Define 
\[x_n = \max\{ [kx]/k : k = 1,2,\dots ,n\} .\]
Note that $x \geq  x_n$ and $x_n \geq [kx]/k$
for every $k \in \{1,\dots,n\}$.  
Therefore
\begin{equation}
\label{eq:x-max}
[kx_n] = [kx] \quad \text{for every} \quad k \in 
 \{1,\dots,n\}.
\end{equation}
Let $d =d_{n,x}$
denote the smallest element of $\{1,2,\dots,n\}$ such that $[dx]/d = x_n$. 
If $y < x_n$, {then} 
${dy < dx_n = [dx]}, $ {so} $ [dy] < [dx]$. 
Thus $x_n$ is the smallest real number satisfying~\eqref{eq:x-max}.

The relation $[nx] \geq \sum^n_{k=1} [kx]/k$ will now
be proved by induction on $n$; 
it obviously holds for all $x$ when $n = 1$. 
Suppose now that $n > 1$ and let $r$ denote the element of
$\{0,1,\dots,d-1\}$ which is congruent to $n$ modulo $d$.  
Observe that $(n-r)x_n$ is an integer, because $n-r$ is a multiple of $d$
and $x_n = [dx]/d$.  
Therefore 
\begin{align}
\label{eq:5}
[r x_n + (n-r) x_n] &= [rx_n] + (n-r)x_n \nonumber \\
&= [rx_n] - \sum^{r}_{k=1} \f{[kx_n]}{k}
+ \sum^{n}_{k=r+1} \f{kx_n-[kx_n]}{k} 
+ \sum^{n}_{k=1} \f{[kx_n]}{k} .
\end{align}
By the induction hypothesis and \eqref{eq:x-max}
we have
$[rx_n] \geq \sum^{r}_{k=1} [kx_{n}]/k$.
Hence \eqref{eq:5}
implies that
\begin{equation}
\label{eq:6}
[nx_n] \geq  
\sum^{n}_{k=r+1} \f{kx_n-[kx_n]}{k} 
+ \sum^{n}_{k=1} \f{[kx_n]}{k} 
\geq \sum^{n}_{k=1} \f{[kx_n]}{k}.
\end{equation}
This relation and \eqref{eq:x-max} imply $[nx] \geq \sum^n_{k=1} [kx]/k$. 
\end{proof}

\section{Preliminary analysis}
\label{sec:prelim}
We now turn toward establishing the main result \eqref{eq:summary}.  
The next result is a partial result in this direction 
and will provide us with some of the background for establishing \eqref{eq:summary}.  
We make use of the same notation as in the proof above, namely
\begin{align*}
x_n &=  \max\{ [kx]/k : k = 1,2,\dots ,n\}, \\
d_{n,x} &= \text{smallest positive integer $d$ such that }[dx]/d = x_n.
\end{align*}
It is clear that $d_{n,x} \leq n$.
It is  shown in the proof below that
 $d_{n,x}$ is the denominator of $x_n$ in lowest terms.

\begin{proposition} 
\label{prop:21}
If $d_{n,x}=1$, then 
$[nx] = \sum^{n}_{k=1} [kx]/k$; 
otherwise, $[nx] \geq
\f{1}{6} + \sum^{n}_{k=1} [kx]/k$.
\end{proposition}

\begin{proof}
Suppose at first that $d = d_{n,x} =1$.  
Then 
$[x] = \max\{[kx]/k : k = 1, 2,\dots,n\}$.  
This observation and the fact that $[kx]/k \geq [x]$ 
for any $x$ and any integer $k \geq 1$
imply that $[kx]/k = [x]$ for every $k \in \{1,\dots,n\}$.  
Hence 
\[
\sum^{n}_{k=1} \f{[kx]}{k} = n[x] = [nx].
\]

Suppose now that $d_{n,x} > 1$, and let $r$ denote (as before) the element of
${ \{0,1,\dots,d-1\} }$ which is congruent to $n$ modulo $d$.  
Statements \eqref{eq:x-max} and \eqref{eq:6} imply that
\begin{equation}
\label{eq:7}
[nx] - \sum^n_{k=1} \f{[kx]}{k} \geq \sum^{n}_{k=r+1} \f{kx_n-[kx_n]}{k}.
\end{equation}
Note that $n \geq r+d$, because $n \geq d > r$ and $n \equiv r \pmod{d}$. 
Therefore 
\begin{align}
\label{eq:8}
\sum^n_{k=r+1} \f{kx_n-[kx_n]}{k} \geq \sum^{r+d}_{k=r+1} \f{kx_n-[kx_n]}{k} 
&\geq \f{1}{r+d} \sum^{r+d}_{k=r+1} (kx_n - [kx_n]) \nonumber \\
&\geq \f{1}{2d-1} \sum^{r+d}_{k=r+1} (kx_n - [kx_n]) .
\end{align}

Observe that, if $k$ is an element of $\{1,\dots,n\}$ such that $kx_n$ is an
integer, then $kx_n = [kx_n] = [kx]$ by \eqref{eq:x-max}, 
so $x_n = [kx]/k$.  
The definition of $d=d_{n,x}$ now implies that $d$ is the smallest positive
integer such that $dx_n$ is an integer.  
Hence $d$ and $dx_n$ are relatively prime.  
Therefore, if $k$ varies over a set of integers which are pairwise
incongruent modulo~$d$, 
then the integers $kdx_n$ will be pairwise incongruent
modulo~$d$, 
and hence the integers $kdx_n - d[kx_n]$ will also be pairwise
incongruent modulo~$d$.  
Since
\[
0 \leq d(kx_n - [kx_n]) < d \quad \text{for any }  k,
\]
we obtain that
\begin{align}
\label{eq:9}
&\text{if $R$ is a set of $d$ integers which are 
      pairwise incongruent modulo $d$,} \nonumber \\
&\text{then } \quad 
 \{kdx_n-d[kx_n]: k \in R\} = \{0,1,\dots,d-1\}. 
\end{align}
A similar observation is made in 
\cite[p. 92]{kl}.  
By \eqref{eq:9},
\begin{equation}
\label{eq:10}
\f{1}{2d-1} \sum^{r+d}_{k=r+1} (kx_n - [kx_n]) =
\f{1}{2d-1} \sum^{d-1}_{k=0} \frac{k}{d}
= \f{d-1}{2(2d-1)}
= \frac{1}{4 + \frac{2}{d-1} }.
\end{equation}
This equation and the supposition that $d=d_{n,x} > 1$ (so $d \ge 2$)
imply that 
\[ \f{1}{2d-1} 
\sum^{r+d}_{k=r+1} {(kx_n-[kx_n])} \geq \f{1}{6}. \]
From \eqref{eq:7} and \eqref{eq:8},
we deduce that $[nx] - \sum^{n}_{k=1} [kx]/k \geq \f{1}{6}$.
\end{proof}

\begin{corollary}
\label{cor:32}
If $x - [x] < \frac{1}{n}$, 
then $[nx] = \sum^{n}_{k=1} [kx]/k$; 
otherwise, 
$[nx] \geq \f{1}{6} + \sum^{n}_{k=1} [kx]/k$.
\end{corollary}
\begin{proof}
By Proposition~\ref{prop:21}
it suffices to show that
\[ x - [x] < \f{1}{n} \qquad\Leftrightarrow\qquad d_{n,x} = 1. \]

Suppose at first that $x - [x] < 1/n$.  Then $kx <
k[x] + k/n \leq k[x] + 1$ for every $k \in 
\{1,\dots,n\}$, and hence $[kx] \leq k[x]$ for every $k 
\in \{1,\dots, n\}$.  Therefore $[x] = \max\{[kx]/k: k =
1,\dots,n\}$ and hence $d_{n,x} = 1$.

Suppose now that $x - [x] \geq {1}/{n}$.  Then $nx \geq n[x] + 1$, so $[nx]
\geq n[x] + 1$.  Hence $[nx]/n > [x]$, so $[x] \neq \max\{[kx]/k: k =
1,\dots,n\}$.  Therefore $d_{n,x} > 1$.
\end{proof}

Note that $[nx] = \f{1}{6} + \sum^{n}_{k=1} [kx]/k$ when
$n = 3 $ and $x = \f{1}{2}$.

\begin{lemma}[Rearrangement inequality]  
\label{lem:3}
Let $b_1,\dots, b_m$ and $c_1,\dots,c_m$ denote real
numbers such that $c_1 > c_2 > \cdots > c_m$.  
Let $\tau$ denote a permutation
of $\{1,\ldots,m\}$ such that 
$b_{\tau(1)} \leq b_{\tau(2)} \leq \cdots \leq
b_{\tau(m)}$.  
Then 
\[
\sum^{m}_{i=1} b_{\tau(i)}c_i \leq \sum^{m}_{i=1} b_i c_i \leq
\sum^{m}_{i=1} b_{\tau(m+1-i)}c_i.
\]
\end{lemma}
This result, and a proof of it, can be found in \cite[p. 261]{hlp}.

\begin{lemma}
\label{lem:4}  
Suppose that $p$ and $q$ are relatively prime
integers and $q \geq 2$.  
Then for every positive integer $n$
\[
\left[ {np}/{q} \right] - \sum^{n}_{k=1} \f{[kp/q]}{k} 
< \left[ {(n+q)p} / {q} \right] - \sum^{n+q}_{k=1} \f{[kp/q]}{k}.
\]
In other words, $f_n(p/q) < f_{n+q}(p/q)$.
\end{lemma}

\begin{proof} 
Let $t = p/q$, and note that $t$ is not an integer. 
This implies 
that $(n+q)t$ and $(n+q-1)t$
cannot both be integers,
so either $[(n+q)t] < (n+q)t$ or $[(n+q-1)t] < (n+q-1)t$ 
(or both).
Thus $[kt]/k < t$ for $k = n + q$ or $n+q-1$; 
note also that $[kt]/k \leq t$ for any $t$ and any $k \geq 1$.  
We deduce that
\[
\sum^{n+q}_{k=n+1} \f{[kt]}{k} < \sum^{n+q}_{k=n+1} t = qt = [nt+qt] - [nt],
\]
where the last equality uses that $qt = p$ is an integer.
Adding $[nt] - \sum^{n+q}_{k=1} [kt]/k$ to both sides
of this relation yields the desired inequality. 
\end{proof}

Recall that $d_{n,x}$ denotes the smallest element 
$d$ of $\{1,\dots,n\}$ such that
$[dx]/d = x_n = \max\{[kx]/k: k = 1,2,\dots,n\}$.

\begin{proposition}
\label{prop:5}
Suppose that $d = d_{n,x}$ satisfies
$[dx] - d[x] \geq 2$; then
\[
[nx] - \sum^{n}_{k=1} \f{[kx]}{k} \geq \f{1}{3}.
\]
\end{proposition}

\begin{proof} 
The supposition that $[dx] - d[x] \geq 2$ implies that
$d(x-[x]) \geq 2$.  
Therefore 
$d \geq 2/(x-[x]) > 2$
since $x-[x] < 1$.  
Hence
\begin{equation}
\label{eq:11}
d \geq 3. 
\end{equation}
Let $r$ denote the element of $\{0,1,\dots,d-1\}$ 
which is congruent to 
$n$ modulo~$d$.  

Suppose at first that $r(x-[x]) \geq 1$, 
so $x - [x] \geq 1/r$.  
Then by  Corollary~\ref{cor:32} (with $n$ replaced by $r$), 
$[rx] \geq \f{1}{6} + \sum^{r}_{k=1} [kx]/k$.  
This observation and \eqref{eq:x-max} imply that 
$[rx_n] \geq \f{1}{6} + \sum^{r}_{k=1} [kx_{n}]/k$.  
Hence, from \eqref{eq:5}, 
$$[nx_n] - \sum^{n}_{k=1} \f{[kx_n]}{k} \geq \f{1}{6} + 
\sum^{n}_{k=r+1} \f{kx_n-[kx_n]}{k}.$$
From \eqref{eq:8}, \eqref{eq:10}, and \eqref{eq:11},
this implies
$$[nx_n] - \sum^{n}_{k=1} \f{[kx_n]}{k} \geq 
\f{1}{6} +  \frac{1}{4 + \frac{2}{d-1}}
\geq \frac{1}{6} + \frac{1}{5} > \frac 1 3.$$ 
This inequality and \eqref{eq:x-max} imply that 
$[nx] - \sum^{n}_{k=1} [kx]/k  > 1/3$.

Suppose now that $r(x-[x]) < 1$.  
This inequality and 
the initial supposition that $[dx] - d[x] \geq 2$
imply that
$d(x-[x]) \geq 2 > 2r(x-[x])$.
Therefore $d > 2r $, so $d \geq 2r + 1$.  
Hence
\begin{equation}
\label{eq:12}
r \leq \left[\f{d-1}{2}\right].
\end{equation}

Inequality \eqref{eq:7} and the first inequality of \eqref{eq:8} imply that
\begin{equation}
\label{eq:13}
[nx] - \sum^{n}_{k=1} \f{[kx]}{k} \geq \sum^{r+d}_{k=r+1} 
\f{kx_n-[kx_n]}{k}.
\end{equation}
We seek a good lower bound for
$\sum^{r+d}_{k=r+1} (kx_n-[kx_n])/k$.  
Statement \eqref{eq:9} implies that
\[
\{kx_n-[kx_n]: k = r+1,\, r + 2, \ldots,r+d\} 
= \left\{\f{0}{d},\,
\f{1}{d}, \ldots ,\, \f{d-1}{d}\right\}.
\]
This observation and Lemma \ref{lem:3},
with $\{b_1,\dots,b_d\} =  \{kx_n-[kx_n]: k = {r+1}, \ldots, {r+d} \}$ 
and 
$\{c_1,\dots , c_d\} = \{\fb{1}{r+1}, \dots, \fb{1}{r+d}\} $,
imply that 
\begin{align*}
\sum^{r+d}_{k=r+1} \f{kx_n-[kx_n]}{k} 
&\geq \f{1}{d} \sum_{j=0}^{d-1}
\frac{j}{j+r+1} .
\end{align*}
Similar inequalities can be found in 
\cite[pp. 92, 93]{kl}.
From \eqref{eq:12}, we obtain
\begin{align}
\label{eq:14}
\sum^{r+d}_{k=r+1} \f{kx_n-[kx_n]}{k}
&\geq \f{1}{d} \sum^{d-1}_{j=0} {\frac{j}{ j+ [(d-1)/2]+1 }} 
\nonumber \\
&= \frac{1}{d} \sum^{d-1}_{j=0} 
\left(1 - \frac{[(d-1)/2] +1}{j+[(d-1)/2]+1}\right) \nonumber \\
&= 1 - \frac{1}{d} \sum_{j=0}^{d-1} \frac{[(d-1)/2]+1}{j + [(d-1)/2] + 1} .
\end{align}

Define $h = [\ft{d-1}{2}] + 1$.  
Suppose at first that $d$ is even.  
Then $d = 2h$ and by \eqref{eq:11} we have
$h  \geq 2$.  From \eqref{eq:14}, we obtain
\begin{align}
\label{eq:15}
\sum^{r+d}_{k=r+1} \f{kx_n-[kx_n]}{k} 
&\geq 1 - \f{1}{2h} \sum^{2h-1}_{j=0} \f{h}{j+h}
\nonumber \\
&= 1 - \f{1}{2} \sum^{3h-1}_{j=h} \f{1}{j} 
\qquad\quad \text{if  $d_{n,x}$ is even}.
\end{align}
Note that
\begin{equation}
\label{eq:16}
\sum^{3m-1}_{j=m} \f{1}{j} \quad 
\text{is a decreasing function of $m$ for $m \geq 1$},
\end{equation}
because 
\[
\sum^{3(m+1)-1}_{j=m+1} \frac1j - \sum^{3m-1}_{j=m} \frac1j
= \f{1}{3m} + \f{1}{3m+1} + \f{1}{3m+2} - \frac1{m} < 0 .\]
Statements \eqref{eq:15} and \eqref{eq:16}, 
together with the fact that $h \geq 2$, 
imply that
\[
\sum^{r+d}_{k=r+1} \f{kx_n-[kx_n]}{k} \geq 1 - \f{1}{2} \sum^{5}_{j=2}
\dfrac{1}{j} = \f{43}{120} > \f{1}{3}.
\]
This inequality and \eqref{eq:13} establish the proposition 
when $d$ is even.

Suppose now that $d$ is odd.  
Then $h = [\ft{d-1}{2}] + 1 = \ft{d+1}{2}$, 
so $d = 2h - 1$.  
From \eqref{eq:14}, we obtain
\begin{align}
\label{eq:17}
\sum^{r+d}_{k=r+1} \f{kx_n-[kx_n]}{k} 
&\geq  1 - \f{1}{2h-1} \sum^{2h-2}_{j=0} \f{h}{j+h} \nonumber \\
&= 1 - \f{h}{2h-1} \sum^{3h-2}_{k=h} \f{1}{k} \nonumber \\
&>  1 - \frac{h}{2h-1}\sum_{k=h}^{3h-1} \frac{1}{k} 
\qquad\quad \text{if $d_{n,x}$ is odd} .
\end{align} 
Note that 
$\frac{h}{2h-1} = \frac12 + \frac1{2(2h-1)}$ 
is a decreasing function of $h$.
Using \eqref{eq:16}, this implies 
\begin{equation}
\label{eq:18}
\frac{h}{2h-1}  \sum^{3h-1}_{k=h} \frac{1}{k}
\quad\text{is a decreasing function of $h$}.
\end{equation}
Statements \eqref{eq:17} and \eqref{eq:18} imply that, 
if $h \geq 5 $, then
\[
\sum^{r+d}_{k=r+1} \f{kx_n-[kx_n]}{k} 
> 1 - \frac{5}{9} \sum^{14}_{k=5} \f{1}{k} \approx 0.351.
\]
This inequality and relation \eqref{eq:13} establish the proposition 
when $d$ is odd and $d \geq 9$.  

One verifies, using \eqref{eq:13} and \eqref{eq:14}, 
that the proposition also holds when 
$d = 5$ or $d = 7$
since
\begin{equation*}
1 - \frac{1}{5} \sum_{j=0}^4 \frac{3}{j+3} \approx 0.344
\qquad\text{and}\qquad
1 - \frac{1}{7} \sum_{j=0}^6 \frac{4}{j+4} \approx 0.374 .
\end{equation*}
To finish the proof, by observation \eqref{eq:11}
it suffices to consider the case that $d = 3$.

Assume that $d = 3$.  
Note that $d(x - [x]) < d = 3$; hence $[dx] - d[x] \leq 2$.  
The initial supposition that $[dx] - d[x] \geq 2$
implies that $[dx] - d[x] = 2$.  
Hence $[dx]/d - [x] = 2/d = 2/3$, so $x_n - [x] = 2/3$.  
From \eqref{eq:x-max}, we obtain
\begin{equation}
\label{eq:19}
[nx] - \sum^{n}_{k=1} \f{[kx]}{k} = [nx_n] - \sum^{n}_{k=1} 
\f{[kx_n]}{k} = [2n/3] - \sum^{n}_{k=1}
\f{\left[2k/3\right]}{k}.
\end{equation}
Observe that $[2n/3] - \sum^{n}_{k=1} \left[2k/3\right]/k \geq 1/3$ 
when $n = 3$, $4$ or $5$.  
This observation and Lemma~\ref{lem:4} 
(with $p/q = 2/3$) imply that 
$[2n/3] - \sum^{n}_{k=1} \left[2k/3\right]/k  \geq 1/3$ 
for all $n \geq 3$.  
Since $n \geq d = 3$, 
statement \eqref{eq:19} implies
the proposition when $d = 3$.
\end{proof}

\section{Smallest limit point of $S$}
\label{sec:limit-point}
In this section we address how
the value of 
 $[nx] - \sum^{n}_{k=1} [kx]/k$,
 for certain $n$ and $x$,
 is related to the series
 $\sum^{\infty}_{k=1} {1}/{(2k(2k+1))}$
 and its partial sums.
\begin{proposition}
\label{prop:31}
Define $\displaystyle \lambda = \sum^{\infty}_{k=1} \dfrac{1}{2k(2k+1)}$.
\begin{enumerate}[(i)]
\item If $d = d_{n,x}$ satisfies 
$[dx] - d[x] = 1$ and $n = 2d-1$, then 
\[
[nx] - \sum^{n}_{k=1} \f{[kx]}{k} = \sum^{d-1}_{k=1} \f{1}{2k(2k+1)}.
\]
\item Suppose that $d = d_{n,x} > 2$.  
If $[dx] - d[x] \neq 1$ or $n \neq 2d-1$, then 
\[
[nx] - \sum^{n}_{k=1} \f{[kx]}{k} > \lambda.
\]
\item  Suppose that $d_{n,x} = 2$.  
If $n = 3$, then 
$[nx] - \sum^{n}_{k=1} [kx]/k = 1/6$. 
If
$n = 5$, then 
$[nx] - \sum^{n}_{k=1} [kx]/k = 4/15$. 
If $n \neq 3$ and $n \neq 5$, then 
$$[nx] -\sum^{n}_{k=1} \frac{[kx]}{k} \geq \frac{71}{210} > \lambda. $$
\end{enumerate}
\end{proposition}

\begin{proof}  
Observe that \eqref{eq:x-max} implies
\begin{equation}
\label{eq:20}
[nx] - \sum^{n}_{k=1} \f{[kx]}{k} = [nx_n] - 
\sum^{n}_{k=1} \f{[kx_n]}{k} 
= [n(x_n - [x])] - \sum^{n}_{k=1}
\f{[k(x_n-[x])]}{k}. 
\end{equation}
One can easily prove by induction on $m$ that
\begin{equation}
\label{eq:21}
1 - \sum^{2m-1}_{k=m} \f{1}{k} = \sum^{m-1}_{k=1} 
\f{1}{2k(2k+1)}.
\end{equation}

Suppose that $[dx] - d[x] = 1$.  
This supposition and the fact that
$x_n =[dx]/d$ imply that $x_n - [x] = 1/d$.  
(It follows that $d\geq 2$.)
Hence, from \eqref{eq:20},
\begin{equation}
\label{eq:22}
[nx] - \sum^{n}_{k=1} \f{[kx]}{k} = \left[ {n}/{d} \right] - 
\sum^{n}_{k=1} \f{[k/d]}{k} 
\qquad \text{provided } [dx]-d[x] = 1.
\end{equation}
If in addition $n = 2d-1$, then by \eqref{eq:21}
\[
[nx] - \sum^n_{k=1} \f{[kx]}{k} = \left[ \f{2d-1}{d} \right] -
\sum^{2d-1}_{k=1} \f{[k/d]}{k} 
= 1 - \sum^{2d-1}_{k=d} \f{1}{k} 
= \sum^{d-1}_{k=1} \f{1}{2k(2k+1)}.
\]
This establishes statement (i) of the proposition.

Observe that
$\sum^{2m-2}_{k=m} 1/k$ is an increasing function of $m\geq 1$, 
because
\[\sum^{2(m+1)-2}_{k=m+1} \frac{1}{k} - \sum^{2m-2}_{k=m} \frac{1}{k}
 = \f{1}{2m-1} + \f{1}{2m} - \f{1}{m} > 0.\]
Therefore
\[
1 - \sum^{2m-2}_{k=m} \f{1}{k} >
{\lim_{m\to\infty}} \left(1 - \sum^{2m-2}_{k=m} 
\f{1}{k}\right)
= {\lim_{m\to\infty}} \left(1 - 
\sum^{2m-1}_{k=m} \f{1}{k}\right).
\]
From \eqref{eq:21}, we deduce
\begin{equation}
\label{eq:23}
1 - \sum^{2m-2}_{k=m} \f{1}{k} > \sum^{\infty}_{k=1} 
\f{1}{2k(2k+1)} = \lambda.
\end{equation}

Now suppose $[dx] - d[x] = 1$ and $n < 2d-1$.
Recall, from the definition of $d = d_{n,x}$, 
that $d \leq n$.  
Therefore, if $n < 2d-1$, then $d \leq n \leq 2d-2$, so $[n/d] = 1$.  
Using \eqref{eq:23} this implies
\[
[n/d] - \sum^{n}_{k=1} \f{[k/d]}{k} = 1 - 
\sum^{n}_{k=d} \f{1}{k} 
\geq 1 - \sum^{2d-2}_{k=d} \f{1}{k}
> 
\lambda
\qquad \text{if } n < 2d-1.
\]
From this inequality and \eqref{eq:22}, we get
\begin{equation}
\label{eq:24}
[nx] - \sum^{n}_{k=1} \f{[kx]}{k} > \lambda \qquad \text{if } \ 
[dx]-d[x] = 1 \text{ and } n < 2d-1.
\end{equation}

As mentioned in the introduction, $\lambda = 1 - \log 2$.  
This can be obtained from \eqref{eq:21} by comparing 
the sum on the left-hand side to an integral.  
We deduce that
\begin{equation}
\label{eq:25}
\lambda = \sum_{k=1}^\infty \frac{1}{2k(2k+1)} < \dfrac{1}{3}.
\end{equation}
This inequality can in fact be  
established without evaluating $\lambda$
explicitly by rewriting the sum defining $\lambda$ as a telescoping series.

Now suppose $[dx] - d[x] = 1$ and $n > 2d-1$.
Observe that if $2d \leq n \leq 3d-1$, then
\begin{equation}
\label{eq:26}
\left[ {n}/{d} \right] - \sum^{n}_{k=1} \f{[k/d]}{k} 
= 2 - \sum^{n}_{k=1} \f{[k/d]}{k}
\geq 2 - \sum^{2d-1}_{k=d} \f{1}{k} - 2 \sum^{3d-1}_{k=2d}
\f{1}{k}.
\end{equation}
Earlier in this proof, 
we showed that $\sum^{2m-2}_{k=m} 1/k \phantom{\Big)}\!\!$ 
is an increasing function of $m\geq 1$.  
A similar argument establishes that 
$\sum^{2m-1}_{k=m} 1/k$ and $\sum^{3m-1}_{k=2m} 1/k$ 
are decreasing functions of $m$. 
This observation and relation \eqref{eq:26} imply that, 
if $2d \leq n \leq 3d-1$ and $d \geq 3$, then 
\[
\left[ {n}/{d} \right] - \sum^{n}_{k=1} \f{[k/d]}{k} 
\geq 2 - \sum^{5}_{k=3} \f{1}{k} - 2 \sum^{8}_{k=6}\f{1}{k} 
= \frac{73}{210} 
 > \frac{1}{3}.
\]
From Lemma~\ref{lem:4} (with $p/q = 1/d$), we deduce
that if $ n \geq 2d$ and $d \geq 3$, then 
$[n/d] - \sum^{n}_{k=1} [k/d]/k > 1/3 $.  
From \eqref{eq:22} and \eqref{eq:25},
\begin{equation}
\label{eq:27}
[nx] - \sum^{n}_{k=1} \f{[kx]}{k} > \lambda \qquad
\text{if } [dx]-d[x] = 1 \text{ and }  n > 2d-1
\text{ and } d > 2.
\end{equation}

Suppose now that $[dx]-d[x] \neq 1$ and $d > 1$.  
The definition of $d=d_{n,x}$ and the
supposition that $d > 1$ imply that $[dx]/d > [x]$.  
Hence $[dx]-d[x] > 0$.  
Since $[dx]-d[x] \neq 1$, we obtain $[dx]-d[x] \geq 2$.  
Therefore, from Proposition~\ref{prop:5} and from~\eqref{eq:25},
\[
[nx] - \sum^{n}_{k=1} \f{[kx]}{k} \geq \f{1}{3} > \lambda
\qquad \text{if } [dx]-d[x] \neq 1 \text{ and } d>1. 
\]
This inequality and relations \eqref{eq:24} and \eqref{eq:27} 
establish statement (ii) of the proposition.

Suppose now that $d = 2$.  
Note that $[dx]-d[x] \leq d(x-[x]) < d = 2$.  
Hence $[dx]-d[x] \leq 1$.  
Note also that, by the definition of $d = d_{n,x}$ and the supposition
that $d \neq 1$, 
we have $[dx]/d > [x]$, so $[dx]-d[x] > 0$.  
Thus, $[dx]-d[x] = 1$. 
From \eqref{eq:22}, we obtain
\begin{equation}
\label{eq:28}
[nx] - \sum^{n}_{k=1} \f{[kx]}{k} = 
\left[ {n}/{2} \right] - \sum^{n}_{k=1} \frac{[k/2]}{k}.
\end{equation}
Note that $[n/2] - \sum^{n}_{k=1} [k/2]/k \geq 71/210 > 1/3$ 
when $n = 6$ or $7$.  
Hence, Lemma~\ref{lem:4} (with $p/q = 1/2)$ implies that 
$[n/2] - \sum^{n}_{k=1} [k/2]/k > 1/3$ 
for all $n \geq 6$.  
Now, \eqref{eq:25} and \eqref{eq:28} establish 
statement (iii) of the proposition for $n \geq 6$.  
Recall that $n \geq d$, so $n \geq 2$.  
One can verify, using \eqref{eq:25} and \eqref{eq:28}, 
that (iii) holds for $n = 2$, $3$, $4$ and $5$.  
Hence it holds for all $n$.  
\end{proof}

Proposition~\ref{prop:31}
implies that, for most pairs $(n,x)$ 
(especially when $n$ is large), 
$[nx] - \sum^{n}_{k=1} [kx]/k > \lambda $.

\section{Proof of main theorem}
\label{sec:main-proof}
We now prove the main result of this paper
(in equivalent form
 \eqref{eq:summary}).

\begin{theorem}
\label{thm:main}
Let $S$ denote the set of numbers of the form
$[nx] - \sum^{n}_{k=1} [kx]/k$, where $x$ varies over all
real numbers and $n$ varies over all positive integers. 
Then 
\begin{align*}
S = &\left( \text{the set of 
partial sums of}\quad 
\sum^{\infty}_{k=1} \f{1}{2k(2k+1)} \right)
\bigcup  \left\{ 0, \f{4}{15} \right\} \\ 
&\bigcup \left( \text{a dense subset of the interval}
\quad\bigg(\sum^\infty_{k=1} \f{1}{2k(2k+1)}, \infty\bigg) \right).
\end{align*}
\end{theorem}

\begin{proof}
Proposition~\ref{prop:21} implies that if $d_{n,x} = 1$, 
then $[nx] - \sum^{n}_{k=1} [kx]/k = 0$, 
and Proposition~\ref{prop:31} implies that if $d_{n,x} \geq 2$, 
then 
$[nx] - \sum^{n}_{k=1} [kx]/k$ equals a partial sum of 
$\sum^{\infty}_{k=1} 1/(2k(2k+1))$ or $4/15$ 
or a number which is strictly greater than
$\sum^{\infty}_{k=1} 1/(2k(2k+1))$.  
Hence 
\begin{align}
\label{eq:29}
S \subseteq \left\{ \text{partial sums of } \sum^{\infty}_{k=1} 
\f{1}{2k(2k+1)} \right\} 
\cup  \left\{0, \f{4}{15}\right\}  
\cup \ \left(\sum^{\infty}_{k=1}
\f{1}{2k(2k+1)}, \infty\right).
\end{align}
It was observed in the introduction that $S$ contains $0$ and $4/15$
and all the partial sums of the series $\sum^{\infty}_{k=1} 1/(2k(2k+1))$.  
This observation and statement \eqref{eq:29} imply that, to finish the proof, 
it suffices to show that $S$
contains a dense subset of the interval
$(\sum^{\infty}_{k=1} 1/(2k(2k+1)), \infty)$.

Let $u$ denote a real number such that 
${u \geq \sum^{\infty}_{k=1} 1/(2k(2k+1)) }$.  
It will be shown that there are elements of
$S$ which are arbitrarily close to $u$.  
Let $t$ denote an integer such that $t \geq 2$.

\medskip \noindent
\textbf{Claim.}
\textit{There is a positive integer $\hat m = \hat m_{u,t}$ such that 
\[
\hat m - \sum^{\hat m t+t-1}_{k=1} \f{[k/t]}{k} < u < 
\hat m - \sum^{\hat m t}_{k=1} \f{[k/t]}{k}.
\]
In other words, $f_{\hat m t + t - 1}(1/t) < u < f_{\hat m t}(1/t)$.}

\begin{proof}[Proof of the claim] 
Observe that by \eqref{eq:21}
\begin{equation}
\label{eq:30}
1 - \sum^{2t-1}_{k=1} \f{[k/t]}{k} = 1 - 
\sum^{2t-1}_{k=t} \f{1}{k} = \sum^{t-1}_{k=1} \f{1}{2k(2k+1)} < u.
\end{equation}
Note that, for every positive integer $m$,
\begin{align*}
m - \sum^{mt+t-1}_{k=1} \f{[k/t]}{k} &= \sum^{m}_{j=1} \left(1 -
\sum^{jt+t-1}_{k=jt} \f{j}{k}\right)
= \sum^{m}_{j=1} \left(\sum^{jt+t-1}_{k=jt} \f{k-jt}{kt}\right) \\
&\geq \sum^{m}_{j=1} \f{1}{(jt+t-1)t} \sum^{jt+t-1}_{k=jt} (k-jt) \\
&= \sum^{m}_{j=1} \f{t-1}{2(jt+t-1)}
= \f{1}{2} \sum^{m}_{j=1} \f{1}{j\f{t}{t-1} +1}
\geq \f{1}{6} \sum^{m}_{j=1} \f{1}{j},
\end{align*}
where in the last step we have used that $t/(t-1) \le 2$ for any $t \ge 2$.
This inequality and the fact that $\sum^{\infty}_{j=1} 1/j$
diverges imply that there are only finitely many positive integers $m$ such that
$m - \sum^{mt+t-1}_{k=1} [k/t]/k < u$.  
Let $\hat m = \hat m_{u,t}$ denote the largest such integer; 
statement \eqref{eq:30} implies that $\hat m$ exists with 
$\hat m \geq 1$.  
The definition of $\hat m$ implies that
\begin{equation}
\label{eq:31}
\hat m - \sum^{\hat m t+t-1}_{k=1} \f{[k/t]}{k} < u \leq \hat m +
1 - \sum^{(\hat m +1)t+t-1}_{k=1} \f{[k/t]}{k}.
\end{equation}
Observe that
\begin{align}
\label{eq:32}
1 - \sum^{(\hat m +1)t+t-1}_{k=\hat mt+1} \f{[k/t]}{k} 
&= 1 - \sum^{ {\hat m}t+t-1}_{k=\hat mt+1} \f{\hat m}{k} 
     - \sum^{({\hat m}+1)t+t-1}_{k=(\hat m+1)t} \f{\hat m+1}{k} \nonumber \\
&= 1 - \sum^{t-1}_{j=1} \f{\hat m}{\hat m t+j} 
    - \sum^{t-1}_{j=0} \f{\hat m+1}{(\hat m+1)t+j} \nonumber \\
&= \sum^{t-1}_{j=1} \left(\f{j/t}{(\hat m+1)t+j} - 
\f{\hat m}{\hat mt+j}\right),
\end{align}
where in the last line we have used that 
\[
1 - \sum^{t-1}_{j=0}
\f{\hat m+1}{(\hat m+1)t+j} = \sum^{t-1}_{j=0} \f{\hat m+1+(j/t)-(\hat
m+1)}{(\hat m+1)t+j}
= \sum_{j=1}^{t-1} \frac{j/t}{(\hat m +1)t +j} .
\]
Since $j/t < 1 \leq \hat m$ for $j < t$, we obtain from \eqref{eq:32} that
\[
1 - \sum^{(\hat m +1)t+t-1}_{k=\hat mt+1} \f{[k/t]}{k} < 0.
\]
Adding a constant to both sides of this inequality yields
\[
\hat m + 1 - \sum^{(\hat m+1)t+t-1}_{k=1} \f{[k/t]}{k} < \hat m 
- \sum^{\hat mt}_{k=1} \f{[k/t]}{k}.
\]
This inequality and relation \eqref{eq:31} establish the claim.
\end{proof}

Note that the distance between two adjacent elements of 
$\{\hat m - \sum^{n}_{k=1} [k/t]/k: 
{n = \hat mt},\, \hat mt+1, \dots , \hat mt+t-1\}$ 
is less than or equal to $\hat m/(\hat m t+1) < 1/t$.  
This observation and the claim imply that there is an
integer $\hat n = \hat n_{u,t}$ such that 
\begin{equation}
\label{eq:33}
\hat mt \leq \hat n \leq \hat mt+t-1 
\quad\text{and}\quad
\left| u - \left(\hat m - \sum^{\hat n}_{k=1} \f{[k/t]}{k}\right) \right| 
< \f{1}{t}.
\end{equation}
Define 
\[ s_{u,t} = f_{\hat n}(1/t) =  [\hat n/t] - \sum^{\hat n}_{k=1} \frac{[k/t]}{k} .\]  
Note that $| u - s_{u,t} | < 1/t$, by \eqref{eq:33},
so $| u - s_{u,t} |$ approaches $0$ as $t$ approaches $\infty$.  
Since $s_{u,t}$ lies in $S$ for
any $t \geq 2$, 
$u$ lies in the closure of $S$.  
This holds for any $u \geq \sum^{\infty}_{k=1} 1/(2k(2k+1))$, 
so $S$ contains a dense subset of $(\sum^{\infty}_{k=1} 1/(2k(2k+1)), \infty)$. 
\end{proof}

\begin{rmk}
The preceding proof and the remark made after statement
\eqref{eq:summary} imply that Theorem~\ref{thm:main} holds true 
when we restrict $x$ in the definition of $S$ 
to be numbers of the form $1/t$ where $t$ is a 
positive integer. 
\end{rmk}

\section{An upper bound for $S_n$}
\label{sec:upper-bound}
Recall that $S_n$ denotes the set of numbers of the form
$ [nx] - \sum^{n}_{k=1} {[kx]}/{k}$
where $x$ varies over all real numbers.
Observe that
\begin{equation*}
[nx] - \sum^{n}_{k=1} \f{[kx]}{k} \leq nx - \sum^{n}_{k=1} 
\f{[kx]}{k} = \sum^{n}_{k=1} \f{kx-[kx]}{k} < \sum^{n}_{k=1} \f{1}{k}.
\end{equation*}
The following theorem sharpens this inequality.
\begin{theorem}
\label{thm:upper-bound}
For fixed $n$ and any value of $x$,
$$[nx] - \sum^{n}_{k=1} \f{[kx]}{k} \leq 
\sum^{n}_{k=2} \f{1}{k}.$$
Equality holds when $x = 1 - \frac1{n}$,
so this bound is sharp
as $x$ varies.
\end{theorem}

\begin{proof}  
We proceed by induction on $n$.  
If $n = 1$, then
$[nx] - \sum^{n}_{k=1} [kx]/k = [x] - [x] = 0$ 
and $\sum^{n}_{k=2} 1/k = 0$.
Therefore the theorem is true when $n = 1$.  

Suppose now that $n > 1$ and define $x_n$ and $d = d_{n,x}$ as in the
beginning of Section~\ref{sec:prelim}.  
Note that
\begin{equation}
\label{eq:34}
[nx_n] = [dx_n + (n-d)x_n] = [dx_n] + [(n-d)x_n], 
\end{equation}
because $dx_n$ is an integer (in fact, $dx_n = [dx]$).

Assume at first that $d < n$.  
From \eqref{eq:x-max} and \eqref{eq:34}, we deduce that
\[
[nx] - \sum^{n}_{k=1} \f{[kx]}{k} = [dx_n] - 
\sum^{d}_{k=1} \f{[kx_n]}{k} + [(n-d)x_n] - \sum^{n}_{k=d+1} 
\f{[kx_n]}{k}.
\]
We use the induction hypothesis to get an upper bound
on the first two expressions on the
right and use that $[(n-d)x_n] \leq (n-d)x_n$ to get an upper on the 
last two expressions.  
We obtain that
\[
[nx] - \sum^{n}_{k=1} \f{[kx]}{k} 
\leq \sum^{d}_{k=2} \f{1}{k} + \sum^{n}_{k=d+1} \f{kx_n-[kx_n]}{k}
< \sum^{n}_{k=2} \f{1}{k}.
\]
This proves the desired bound when $d < n$.
(In this case, the bound is strict.)

Assume now that $d = n$.  
From relation~\eqref{eq:x-max} we have
\begin{align}
\label{eq:35}
[nx] - \sum^{n}_{k=1} \f{[kx]}{k} 
= [nx_n] - \sum^{n}_{k=1} \f{[kx_{n}]}{k}
\leq \sum^{n}_{k=1} \f{kx_n-[kx_n]}{k}.
\end{align}
The assumption that ${d=n}$ and statement \eqref{eq:9} imply that
$\{ {kx_n-[kx_n]} : k = 1,2,\dots,n\} = \{0/n, 1/n, \dots, {(n-1)/n} \}$. 
From Lemma~\ref{lem:3} 
(with $b_k = {kx_n-[kx_n]}$ and $c_k = 1/k$) 
we obtain
\begin{equation}
\label{eq:36}
\sum_{k=1}^n \f{kx_n-[kx_n]}{k} 
\leq  \sum_{k=1}^{n} \frac{n-k}{n}\cdot \frac{1}{k}
= \sum_{k=1}^n \left( \frac{1}{k} - \frac{1}{n}\right)
= \sum_{k=2}^n \frac{1}{k}.
\end{equation}
Relations~\eqref{eq:35} and \eqref{eq:36} imply 
the desired bound when $d = n$.

If $x = 1 - {1}/{n}$ it is straightfoward to verify that
$[nx] - \sum^{n}_{k=1} [kx]/k = \sum^{n}_{k=2} 1/k$. 
\end{proof}
It can be shown that the relation in Theorem~\ref{thm:upper-bound} is an equality 
if and only if $x -[x] \geq 1-1/n$.  
We omit the details.

\begin{rmk}
Note that the upper bound in Theorem~\ref{thm:upper-bound} is 
\[ H_n - 1 = \log n + \gamma - 1 + o(1) 
\qquad \text{as } n\to \infty
\]
where $H_n$ is the $n$-th harmonic number
and $\gamma\approx 0.577$ is the Euler--Mascheroni constant.
\end{rmk}

\section*{Acknowledgements}
I am grateful to Charles Nicol for bringing to my attention
the inequality \eqref{eq:olympiad} and encouraging me to prove it.  
I am also grateful to John Selfridge for describing his proof 
of \eqref{eq:olympiad} to me and telling me that \eqref{eq:olympiad}
was in a U.S.A. Mathematical Olympiad.

\end{document}